\documentclass[11pt]{article}

\usepackage{amsmath,amssymb,amsthm,mathtools,hyperref,tikz}
\usepackage[margin=1in]{geometry}

\newtheorem{theorem}{Theorem}
\newtheorem{lemma}[theorem]{Lemma}
\newtheorem{proposition}[theorem]{Proposition}

\newtheoremstyle{boldnumber}
  {3pt}        
  {3pt}        
  {\normalfont}
  {}           
  {\bfseries}  
  {.}          
  {.5em}       
  {}           

\theoremstyle{boldnumber}
\newtheorem{remark}[theorem]{Remark}

\newcommand{\La}{\operatorname{La}}
\newcommand{\PG}{\operatorname{PG}}
\newcommand{\F}{\mathcal F}
\newcommand{\G}{\mathcal G}
\newcommand{\A}{\mathcal A}
\newcommand{\B}{\mathcal B}
\newcommand{\Span}{\operatorname{span}}
\newcommand{\Ftwo}{\mathbb F_2}
\newcommand{\Fq}{\mathbb F_q}

\title{Crown-free families and forbidden subposets with $e(P)\in \{1,2\}$}

\author{Bal\'azs Patk\'os\thanks{HUN-REN Alfr\'ed R\'enyi Institute of Mathematics and Department of Computer Science and Information Theory, Budapest University of Technology and Economics, email: patkos@renyi.hu} \and Casey Tompkins\thanks{email: casey.tompkins@renyi.hu}}
\date{}

\begin{document}

\maketitle

\begin{abstract}
The maximum size of a weak $P$-free family $\mathcal{F}\subseteq 2^{[n]}$ is denoted by $\La(n,P)$. Let $e(P)$ denote the maximum integer $k$ such that the union of any $k$ consecutive layers of $2^{[n]}$ is weak $P$-free. In recent years, multiple examples of posets with $e(P)<\pi^-(P):=\liminf_{n\to\infty}
\frac{\La(n,P)}{\binom{n}{\lfloor \frac{n}{2}\rfloor}}$ have been found.
We add several further posets with $e(P)=1$ to this list. We define a family $\F\subseteq 2^{[n]}$ of size at least $(1.22+o(1))\binom{n}{\lfloor \frac{n}{2}\rfloor}$ that is weak $O_6$-free, where $O_6$ is the six-element crown poset. 
We also show an infinite set of posets $P$ with $1=e(P)<\pi^-(P)$ that are minimal with respect to this property.

Finally, we consider how far apart $e(P)$ and $\pi^-(P)$ can be. We prove that for
every fixed finite poset $P$ with $e(P)=1$, there is a constant
$\delta_P>0$ such that
$\La(n,P)\le(2-\delta_P+o(1))\binom{n}{\lfloor n/2\rfloor}$. The value 2 is optimal: explicit vertex-edge incidence posets with $e(P)=1$
have $\pi^-(P)$ values tending to $2$. In contrast, for every $K>0$ we
construct a finite poset $P$ with $e(P)=2$ and lower density greater than
$K$. 
\end{abstract}
\section{Introduction}
For a set $X$, we write $2^X=\{Y:Y\subseteq X\}$ and $\binom{X}{s}=\{Y\subseteq X:|Y|=s\}$.

Many problems in extremal finite set theory address the maximum or minimum size of a set family that does not contain some prescribed intersection or inclusion pattern. A general framework for problems involving a particular inclusion pattern was introduced by Katona and Tarj\'an \cite{KatonaTarjan} in the early 1980s. A family $\G$ of sets is a \textit{weak copy} of a poset $(P,\leqslant)$ if there exists a bijection $\psi:P\rightarrow \G$ such that $\psi(p)\subseteq \psi(p')$ whenever $p\leqslant p'$. A family $\F$ is weak $P$-free if it does not contain a weak copy of $P$. The maximum possible size of a weak $P$-free family $\F\subseteq 2^{[n]}$ is denoted by $\La(n,P)$. As a chain $C_{|P|}$ of length $|P|$ is a weak copy of $P$, by an old result of Erd\H os \cite{E}, we have $\La(n,P)\le (|P|-1)\binom{n}{\lfloor \frac{n}{2}\rfloor}$ and if $P$ is not an antichain, then $\binom{[n]}{\lfloor \frac{n}{2}\rfloor}$ shows $\La(n,P)\ge \binom{n}{\lfloor \frac{n}{2}\rfloor}$ and so $\La(n,P)=\Theta(\binom{n}{\lfloor \frac{n}{2}\rfloor})$.  As opposed to graph and hypergraph Tur\'an problems, the limit $\pi(P)=\lim_{n \to \infty}\frac{\La(n,P)}{\binom{n}{\lfloor \frac{n}{2}\rfloor}}$ is not known to exist for all posets $P$, and so we write
\[
\pi^-(P)=\liminf_{n\to \infty}\frac{\La(n,P)}{\binom{n}{\lfloor \frac{n}{2}\rfloor}} \quad \text{and}\quad \pi^+(P)=\limsup_{n\to \infty}\frac{\La(n,P)}{\binom{n}{\lfloor \frac{n}{2}\rfloor}}.
\]
For more on the history of forbidden subposet problems, see the surveys \cite{AMP,GriggsLi} and Chapter 7 of \cite{GerbnerPatkos}.

A natural way to construct $P$-free families is to consider consecutive layers of $2^{[n]}$. Let $e(P)$ be the maximum integer $k$ such that, for every $n$, the union of any $k$ consecutive layers of $2^{[n]}$ is $P$-free. Clearly, $e(P)\le \pi^-(P)$ for any poset $P$.

Ellis, Ivan, and Leader \cite{EllisIvanLeader} were the first to prove that $d=e(B_d)<\pi^-(B_d)$ for all $d\ge 4$, where $B_d=(2^{[d]},\subseteq)$ is the Boolean lattice. The same holds for $d=2,3$ as shown in \cite{Tompkins} and \cite{Patkos}. The inequality $e(P)<\pi^-(P)$ was established for posets related to $B_d$ and for generalized diamonds \cite{GerbnerPatkos2, Patkos}.

In the present paper, we address the problem of how different $e(P)$ and $\pi^-(P)$ can be with special focus on the cases $e(P)$ being 1 or 2. There is a special class of posets satisfying $e(P)=1$. For any simple graph $G=(V,E)$, we define its vertex-edge incidence poset $P_G=(V\cup E,<)$ with $v<e$ if and only if $v\in e$. The posets $P_{C_k}$ are known as the \textit{crown posets}, denoted by $O_{2k}$, where $C_k$ is the cycle of length $k$. Equivalently, $O_{2k}$ has $2k$ elements $a_1,\dots,a_k,b_1,\dots,b_k$ with $a_i<b_j$ if and only if $j=i,i+1$ with addition taken modulo $k$. With the latter definition $O_4$ is also defined, but $e(O_4)=2$. It was proved by Griggs and Lu \cite{GriggsLu} that $\pi(O_{2k})=1$ for all even $k\ge 4$ and Lu proved \cite{Lu} the same conclusion for odd $k\ge 7$. ($\pi(O_4)=2$ was shown by De~Bonis, Katona, and Swanepoel \cite{DKS}.) We will show that $\pi^-(O_6)>1$, so $O_{10}$ remains the only crown poset for which it is not known whether $\pi^-(O_{10})=e(O_{10})$. We will, however, give examples of posets $P$ for which $\pi^-(P)>1$, $P$ contains $O_{10}$ but not $O_6$ and $e(P)=1$ (one such example, $P_{\theta_1}$, is pictured in Figure~\ref{f1}).

\begin{figure}[ht]
\centering

\begin{tikzpicture}[
    scale=0.85,
    every node/.style={circle, fill=black, inner sep=1.7pt}
]

\node (a1) at (0,0) {};
\node (a2) at (1.4,0) {};
\node (a3) at (2.8,0) {};

\node (b1) at (0.7,1.4) {};
\node (b2) at (2.1,1.4) {};
\node (b3) at (3.5,1.4) {};

\draw (a1)--(b1);
\draw (a2)--(b1);
\draw (a2)--(b2);
\draw (a3)--(b2);
\draw (a3)--(b3);
\draw (a1)--(b3);

\node[draw=none,fill=none] at (1.75,-0.55) {$O_6$};

\end{tikzpicture}
\hspace{0.8cm}
\begin{tikzpicture}[
    scale=0.85,
    every node/.style={circle, fill=black, inner sep=1.7pt}
]

\node (a1) at (0,0) {};
\node (a2) at (1.1,0) {};
\node (a3) at (2.2,0) {};
\node (a4) at (3.3,0) {};
\node (a5) at (4.4,0) {};

\node (b1) at (0.55,1.4) {};
\node (b2) at (1.65,1.4) {};
\node (b3) at (2.75,1.4) {};
\node (b4) at (3.85,1.4) {};
\node (b5) at (4.95,1.4) {};

\draw (a1)--(b1);
\draw (a2)--(b1);
\draw (a2)--(b2);
\draw (a3)--(b2);
\draw (a3)--(b3);
\draw (a4)--(b3);
\draw (a4)--(b4);
\draw (a5)--(b4);
\draw (a5)--(b5);
\draw (a1)--(b5);

\node[draw=none,fill=none] at (2.5,-0.55) {$O_{10}$};

\end{tikzpicture}
\hspace{0.8cm}
\begin{tikzpicture}[
    scale=0.85,
    every node/.style={circle, fill=black, inner sep=1.7pt}
]

\node (a0) at (-1.1,0) {};
\node (a1) at (0,0) {};
\node (a2) at (1.1,0) {};
\node (a3) at (2.2,0) {};
\node (a4) at (3.3,0) {};
\node (a5) at (4.4,0) {};

\node (b0) at (-0.55,1.4) {};
\node (b1) at (0.55,1.4) {};
\node (b2) at (1.65,1.4) {};
\node (b3) at (2.75,1.4) {};
\node (b4) at (3.85,1.4) {};
\node (b5) at (4.95,1.4) {};

\draw (a0)--(b0);
\draw (a1)--(b0);
\draw (a0)--(b3);

\draw (a1)--(b1);
\draw (a2)--(b1);
\draw (a2)--(b2);
\draw (a3)--(b2);
\draw (a3)--(b3);
\draw (a4)--(b3);
\draw (a4)--(b4);
\draw (a5)--(b4);
\draw (a5)--(b5);
\draw (a1)--(b5);

\node[draw=none,fill=none] at (2.15,-0.6)
    {$P_{\theta_1}$};

\end{tikzpicture}

\caption{Hasse diagrams of $O_6$, $O_{10}$, and $P_{\theta_1}$.}\label{f1}
\end{figure}
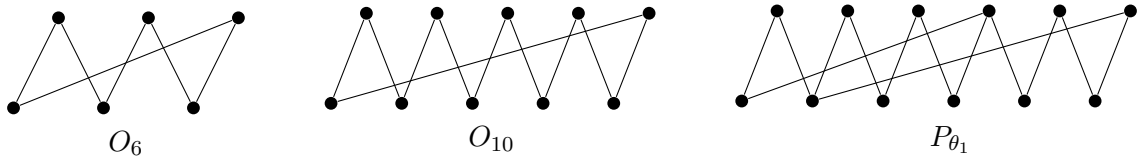 

The Hasse diagram of a poset $P$ is the graph with vertex set $P$ in which $p,q\in P$ are joined by an edge if $p<q$ but there exists no $z\neq p,q$ with $p<z<q$. 


\begin{theorem}\label{o6theta} Write $\Delta_q=\prod_{i=1}^\infty(1-q^{-i})$. Then
\[
\La(n,O_6) 
\ge
\left(1+c+o(1)\right)
\binom n{\lfloor n/2\rfloor},
\]
where
\[
c
=
\Delta_2
\sum_{j=1}^{\infty}
\frac{2^{-j}}
{\left(\prod_{i=1}^{j-1}(2^i-1)\right)^2}
>
0.22.
\]
\end{theorem}

A poset is a \textit{tree poset} if its Hasse diagram is a tree. A theorem of Bukh~\cite{Bukh} states that for any \textit{tree poset}~$T$, $\La(n,T)=(h(T)-1+o(1))\binom{n}{\lfloor \frac{n}{2}\rfloor}$, where $h(T)$ is the \textit{height} of $T$, the size of the largest chain in $T$. If $e(P)=1$, then $h(P)= 2$ as a chain of length 3 cannot be embedded into 2 consecutive levels. As every proper weak subposet of $O_6$ is weakly contained in a tree poset of height two, Bukh's result implies that $O_6$ is minimal among the posets satisfying $1=e(P)<\pi^-(P)$. Our next theorem states that the set of such minimal posets is infinite.

\begin{theorem}\label{infiniteminimal}
  There exist infinitely many pairwise nonisomorphic posets that are
minimal, under weak subposet containment, among the posets $P$
satisfying
\[
  e(P)=1\qquad\text{and}\qquad \pi^-(P)>1.
\] 
\end{theorem}
 
Then we prove that $\pi^+(P)$ is at most $2-\delta_P$ for some positive $\delta_P$ for all posets with $e(P)=1$. Note that the result without the $\delta_P$ immediately follows from the above theorem of Bukh.  Indeed, any weak copy of $K_{s,1,t}$ contains a weak copy of $K_{s,t}$, which in turn contains a weak copy of $P$, where $K_{s,t}$ is the complete bipartite poset with $s$ minimal and $t$ maximal elements and $K_{s,1,t}$ is the complete tripartite poset obtained from $K_{s,t}$ by adding a new element between the minimal and maximal elements of $K_{s,t}$. $K_{s,1,t}$ is a tree poset, so Bukh's result shows $\pi^+(P)\le \pi(K_{s,1,t})\le 2$. For posets $P_G$, Griggs and Lu \cite{GriggsLu} proved the conclusion with a better value of $\delta_P$. Some of the ideas used in the proof are from their paper. 

\begin{theorem}\label{thm:e1-gap}
For every fixed finite poset $P$ with $e(P)=1$, there is a constant
$\delta_P>0$ such that
\[
  \pi^+(P)\le2-\delta_P.
\]
More explicitly, if $P$ embeds into levels $\ell,\ell+1$ of $B_d$,
then one may take
\[
  \delta_P=\frac1{2\left(\binom d\ell+\binom d{\ell+1}-1\right)}.
\]
\end{theorem}

Next we consider posets with $e(P)=2$ and show that $\pi^-(P)$ can be arbitrarily large. The same conclusion holds for every $k\ge 3$.  For the value of $m$
given below, let $P$ be obtained from $S_2(m,0)$ by adjoining a chain
of $k-2$ new elements below its unique minimum.  Then $h(P)=k+1$, so
$e(P)\ge k$, while the natural embedding into the bottom $k+1$ levels
of a Boolean lattice gives $e(P)\le k$.  Moreover, $S_2(m,0)$ is a
weak subposet of $P$, and hence
\[
\pi^-(P)\ge \pi^-\bigl(S_2(m,0)\bigr).
\] For $m\ge2$, define the lower three-level poset
$  S_2(m,0)=\{X\subseteq[m]:|X|\le2\}$,
ordered by inclusion. Observe that $e\bigl(S_2(m,0)\bigr)=2$ for all $m$.   The poset contains a three-element chain, so it
cannot occur in two levels.  Conversely, the bottom three levels of
$B_m$ contains the natural copy $X\mapsto X$ for $|X|\le2$, so the
union of three consecutive levels need not be $S_2(m,0)$-free. Note that two consecutive middle layers together with the $(r,m)$-daisy free construction of Ellis, Ivan, and Leader from \cite{EllisIvanLeader} on the next layer already showed that $\pi^-(S_2(m,0))>2=e(S_2(m,0))$ for all $m\ge 4$. The following theorem shows that $\pi^-(S_2(m,0))$ tends to infinity with $m$.

\begin{theorem}\label{thm:e2}
For every $K>0$ there is an $m$ such that
$ \pi^-\bigl(S_2(m,0)\bigr)>K$.
One may take any integer $L>2K$ and then $m=2^{L^2}+1$.
\end{theorem}

\section{Preliminaries from linear algebra}
All our constructions use linear algebra. In this section we gather some basic facts that will be helpful in later proofs.

\begin{proposition}[Proposition 3 \cite{Tompkins}]\label{propRankcount}
Let $q$ be a prime power and let $0\le r\le \min\{m,s\}$.  The number of $m\times s$ matrices over $\Fq$ having rank $r$ is
\begin{equation}\label{eqRankcount}
\genfrac{[}{]}{0pt}{}{m}{r}_q\prod_{i=0}^{r-1}(q^s-q^i).
\end{equation}
\end{proposition}

\begin{lemma}\label{lem:random}
Let $v_1,\dots,v_s$ be chosen independently and uniformly from $\Fq^d$, where $s\le d$.
Then
\[
  \Pr(v_1,\dots,v_s\text{ are dependent})
  \le \sum_{i=0}^{s-1}q^{i-d}.
\]
\end{lemma}

\begin{proof}
If the first $i$ vectors are independent, their span contains $q^i$ of the
$q^d$ possible values of $v_{i+1}$.  Thus the probability that the first
failure occurs at that step is at most $q^{i-d}$.  Sum over the possible first
failure steps.
\end{proof}

We also make use of the following simple estimate of the number of superspaces of a given subspace.

\begin{lemma}\label{lem:superspaces}
Let $W\le V=\mathbb F_q^d$ have codimension $r$.  There are $\sum_{i=0}^r{r\brack i}_q\le q^{r^2}$
subspaces $H$ satisfying $W\le H\le V$.
\end{lemma}

\begin{proof}
    There is a one-to-one correspondence between subspaces satisfying $W\le H \le V$ and subspaces of $V/W$. Since $\frac{{r\brack j-1}_q}{{r\brack j}_q}\le \frac{1}{q}$ for $j<\frac{r}{2}$, their number is $$\sum_{i=0}^r{r\brack i}_q\le 2\sum_{i=0}^{\lfloor \frac{r}{2}\rfloor}{r \brack i}_q\le \frac{2q}{q-1}{r\brack \lfloor \frac{r}{2}\rfloor}_q\le \frac{2q}{q-1}q^{\lfloor \frac{r}{2}\rfloor(r-\lfloor \frac{r}{2}\rfloor+1)}\le q^{\lfloor \frac{r}{2}\rfloor(r-\lfloor \frac{r}{2}\rfloor+1)+2}\le q^{r^2}$$ if $r\ge 2$. The $r=1$ case is $2\le q= q^{r^2}$.
\end{proof}



\section{Proof of Theorem \ref{o6theta}}
In this section we prove Theorem \ref{o6theta}. First we describe our construction. Put
$k=\left\lfloor\frac n2\right\rfloor$, 
$V=\Ftwo^{k+1}$,
and fix a nonzero vector
$u\in V$.
Assign an arbitrary label
$\phi(x)\in V$
to every $x\in[n]$; repeated labels are allowed.  For
$S\subseteq[n]$, write
\[
W(S)=\Span\{\phi(x):x\in S\}.
\]

Define
\[
\A=
\left\{
A\in\binom{[n]}k:
u\notin W(A)
\right\}
\hskip 0.5truecm
\B=
\left\{
B\in\binom{[n]}{k+1}:
u\in W(B)
\right\},
\hskip 0.5truecm
\F=\A\cup\B.
\]
We shall prove that, for every choice of the labeling $\phi$,
the family $\F$ is $O_6$-free. Then we calculate the expected size of $\F$. The lower bound of Theorem \ref{o6theta} will then follow as some labeling $\phi$ produces a family $\F$ of at least  expected size.

\bigskip

We start by showing the $O_6$-free property. Recall that $O_6$ has three minimal elements
$a_1,a_2,a_3$ and three maximal elements
$b_{12},b_{13},b_{23}$, where
$a_i<b_{rs}$ if and only if $i\in\{r,s\}$.
For
$C\in\binom{[n]}{k-1}$,
define the \textit{selected link graph} $G_C$ as follows.  Its vertex set is
$V(G_C)=
\{x\in[n]\setminus C:C\cup\{x\}\in\A\}$,
and two distinct vertices $x,y\in V(G_C)$ are adjacent if
$
C\cup\{x,y\}\in\B$.

\begin{lemma}\label{O6triangle}
The family $\A\cup\B$ contains a weak copy of $O_6$ if and only if
$G_C$ contains a triangle for some
$C\in\binom{[n]}{k-1}$.
\end{lemma}

\begin{proof}
Suppose first that $x_1x_2x_3$ is a triangle in $G_C$.
Map $a_i\to C\cup\{x_i\}$
and $b_{ij}\to C\cup\{x_i,x_j\}$.
These six sets are distinct and preserve every required order
relation, so they form a weak copy of $O_6$.

Conversely, suppose that $\A\cup\B$ contains a weak copy of $O_6$.
Let $A_i$ be the image of $a_i$, and let $B_{ij}$ be the image of
$b_{ij}$.  Since $\F$ occupies only levels $k$ and $k+1$, we have
$|A_i|=k, |B_{ij}|=k+1$.
For every $ij\in\{12,13,23\}$,
$A_i,A_j\subset B_{ij}$.
Since $A_i$ and $A_j$ are distinct $k$-subsets of the
$(k+1)$-set $B_{ij}$,
$|A_i\cap A_j|=k-1$
and
$B_{ij}=A_i\cup A_j$.
Thus either $A_1,A_2,A_3$ have a common $(k-1)$-subset, or they
are three different $k$-subsets of one $(k+1)$-set.
The second possibility cannot occur here.  Indeed, in that case
$A_1\cup A_2=A_1\cup A_3=A_2\cup A_3$ implying
$B_{12}=B_{13}=B_{23}$,
contrary to the injectivity of the copy.
Hence there exists
$C\in\binom{[n]}{k-1}$
and distinct $x_1,x_2,x_3\notin C$ such that
$A_i=C\cup\{x_i\}$.
Then by the above
$B_{ij}=C\cup\{x_i,x_j\}$.
Therefore $x_1x_2x_3$ is a triangle in $G_C$.
\end{proof}

\begin{lemma}\label{fibers}
For every
$C\in\binom{[n]}{k-1}$,
the graph $G_C$ is a disjoint union of complete bipartite graphs
and isolated vertices.
\end{lemma}

\begin{proof}
Fix $C$ and put
$W=W(C)$.
If $u\in W$, then
$u\in W(C\cup\{x\})$
for every $x\notin C$, so $G_C$ has no vertices.
Suppose therefore that
$u\notin W$.
Pass to the quotient
$Q=V/W$.
Write
\[
\bar u=u+W\neq0,
\qquad
v_x=\phi(x)+W.
\]
An element $x\notin C$ is a vertex of $G_C$ precisely when
$\bar u\notin\Span\{v_x\}$.
Over $\Ftwo$,
$\Span\{v_x\}=\{0,v_x\}$,
so $x\notin C$ is a vertex of $G_C$ if and only if
$v_x\neq\bar u$. For two such $x,y$,
\begin{align*}
xy\in E(G_C)
\iff
u\in W(C\cup\{x,y\})
\iff
\bar u\in\Span\{v_x,v_y\}.
\end{align*}
If $v_x$, $v_y$ are linearly dependent, then any nonzero vector in the span of $\{v_x,v_y\}$ is equal to $v_x$ or $v_y$.
Since neither \(v_x\) nor \(v_y\) equals \(\bar u\), their span does not
contain \(\bar u\).
If they are independent, their two-dimensional
binary span is
$\{0,v_x,v_y,v_x+v_y\}$.
This contains $\bar u$ precisely when
$v_y=v_x+\bar u$.

The involution
$v\to v+\bar u$
partitions
$Q\setminus\{0,\bar u\}$
into pairs
$\{v,v+\bar u\}$.
By the above, for every such pair all possible edges join the
fiber of $v$ to the fiber of $v+\bar u$, giving a complete
bipartite component.  There are no edges between different such
pairs.  The $0$-fiber is isolated, while elements of the $\bar u$-fiber are not in $V(G_C)$. Thus $G_C$ is a disjoint union of complete bipartite graphs and
isolated vertices.
\end{proof}

By Lemma \ref{fibers}, every $G_C$ is triangle-free, so Lemma \ref{O6triangle} yields
the following.

\begin{proposition}
For every labeling $\phi$, the family $\F$ is $O_6$-free.
\end{proposition}

\bigskip

Next we calculate the expected size of our construction. We now choose the labels
$
\phi(x)$, $x\in[n]$,
independently and uniformly at random from
$
V=\Ftwo^{k+1}$.
Let $X_1,\dots,X_{k+1}$
be independent uniform random vectors in $V$, and put
$
W_{k}=\Span\{X_1,\dots,X_{k}\}$,
 $W_{k+1}=\Span\{X_1,\dots,X_{k+1}\}$.

For a fixed $k$-set and a fixed $(k+1)$-set, respectively, their
probabilities of belonging to $\A$ and $\B$ are
\[
L_{k+1}=\Pr(u\notin W_k),
\qquad
U_{k+1}=\Pr(u\in W_{k+1}).
\]
Since
$u\in W_{k}$ implies $u\in W_{k+1}$,
the events
$\{u\notin W_{k}\}$ and
$\{u\in W_{k+1}\}$
cover the whole probability space.  Consequently,
\[
L_{k+1}+U_{k+1}=1+c_{k+1},
\]
where
$c_{k+1}
=
\Pr(u\notin W_{k},\,u\in W_{k+1})$.
We calculate the limit of $c_{k+1}$.

Put
$J=k+1-\dim W_{k}$. 
Since $W_{k}$ is spanned by $k$ vectors,
$1\le J\le k+1$.
Write
$\rho_{k+1,j}=\Pr(J=j)$.
Proposition \ref{propRankcount} yields
\begin{equation}\label{prob}
    \rho_{k+1,j}
=
2^{-j(j-1)}
\frac{
\displaystyle
\prod_{\ell=j+1}^{k+1}(1-2^{-\ell})
\prod_{\ell=j}^{k}(1-2^{-\ell})
}{
\displaystyle
\prod_{\ell=1}^{k+1-j}(1-2^{-\ell})
}.
\end{equation}

Conditional on $J=j$, the subspace $W_{k}$ is uniformly
distributed among all $(k+1-j)$-dimensional subspaces of $V$.
For a fixed nonzero vector $u$,
\[
\Pr(u\in W_{k}\mid J=j)
=
\frac{2^{k+1-j}-1}{2^{k+1}-1}.
\]
Therefore
\[
\Pr(u\notin W_{k}\mid J=j)
=
\frac{2^{k+1}-2^{k+1-j}}{2^{k+1}-1}
=
\frac{1-2^{-j}}{1-2^{-(k+1)}}.
\]

If $u\notin W_{k}$, then
$u\in W_{k+1}$
exactly when
$X_{k+1}\in u+W_{k}$.
The coset $u+W_{k}$ has $2^{k+1-j}$ elements, and hence
$\Pr(u\in W_{k+1}\mid
J=j,\ u\notin W_{k})
=
2^{-j}$. We obtain
\[
c_{k+1}
=
\frac{1}{1-2^{-(k+1)}}
\sum_{j=1}^{k+1}
\rho_{k+1,j}(1-2^{-j})2^{-j}.
\]

For $t\ge1$, put
$P_t=\prod_{i=1}^t(1-2^{-i}),
P_0=1$,
and let
$\Delta_2
=
\prod_{i=1}^{\infty}(1-2^{-i})$.
For every fixed $j$, (\ref{prob}) gives
\[
\rho_{k+1,j}
\longrightarrow
q_j
=
\frac{2^{-j(j-1)}\Delta_2}{P_jP_{j-1}}.
\]
Moreover,
\[
0\le\rho_{k+1,j}
\le
\Delta_2^{-1}2^{-j(j-1)},
\]
which is summable in $j$.  Dominated convergence of the $\rho_{k+1,j}$s therefore
yields
\[c
:=
\lim_{k\to\infty}c_{k+1}
=
\sum_{j=1}^{\infty}
q_j(1-2^{-j})2^{-j}
=
\Delta_2
\sum_{j=1}^{\infty}
\frac{2^{-j}}
{\left(\prod_{i=1}^{j-1}(2^i-1)\right)^2}.
\]The first terms of the series multiplying $\Delta_2$ are
$\frac12+\frac14+\frac1{72}+\frac1{7056}+\cdots$,
and numerically
$
c>0.22$.

By linearity of expectation,
$\mathbb E|\F|
=
L_{k+1}\binom nk
+
U_{k+1}\binom n{k+1}$.
Since
$\binom n{k+1}=(1+o(1))\binom nk$,
we obtain
$\mathbb E|\F|
=
(1+c+o(1))\binom nk$.

\begin{remark}
The construction also shows that $1=e(P)<\pi^-(P)$ for some other posets that do not contain $O_6$ as a weak subposet. Consider the posets $P_{\theta_1}, P_{\theta_2}$ whose Hasse diagrams are the bipartite graphs $\Theta(3,5,5)$ and $\Theta(5,5,5)$, respectively.\footnote{The $\Theta$-graph $\Theta(i,j,k)$ has two vertices joined by three internally vertex-disjoint paths of lengths $i,j$, and $k$, respectively.} The Hasse diagram of $P_{\theta_1}$ is shown in Figure 1. These posets embed into two consecutive levels: for example for $P_{\theta_1}$ one can consider six pairs $bt,gt,dt,at,ab,ag$
and six triples $bgt,agt,bdt,adt,abt,abg$. As the cycle lengths in $\Theta(3,5,5)$ and $\Theta(5,5,5)$ are 8 and 10 and only 10, respectively, none of $P_{\theta_1},P_{\theta_2}$ contains $O_6$. On the other hand, it is not hard to see that if a copy $\G$ of any of these posets lives on two consecutive levels of $2^{[d]}$ for some $d$, then $\G$ does contain a copy of $O_6$. Since our construction lives on two consecutive levels and is $O_6$-free, it is also $\{P_{\theta_1},P_{\theta_2}\}$-free. $P_{\theta_1},P_{\theta_2}$ both contain $O_{10}$, and have only 2 and 4 extra elements compared to $O_{10}$. Thus $\pi^-(P_{\theta_1})>1$ and $\pi^-(P_{\theta_2})>1$, which suggests that $\pi^-(O_{10})>1$, although we have not been able to prove this claim. 
\end{remark}

\section{An infinite set of minimal posets with $1=e(P)<\pi^-(P)$}

In this section we prove Theorem \ref{infiniteminimal}. We start with an observation on the intersection and union of $r$-sets. We leave the proof to the reader.

\begin{lemma}\label{lem:block}
Let $X_0,X_1,Y_0,Y_1$ be four distinct $r$-subsets of a set, and
suppose that  $|X_i\cap Y_j|=r-1$ $(i,j\in\{0,1\})$.
Suppose in addition that the four sets $X_i\cup Y_j$ are distinct.
Then exactly one of the following holds.
\begin{enumerate}
  \item There is an $(r-1)$-set $R$ contained in all four sets.  More
  precisely, for four distinct elements $a_0,a_1,b_0,b_1\notin R$, $X_i=R\cup\{a_i\}, ~Y_j=R\cup\{b_j\}$.
  We call this the \emph{star type}.

  \item There are an $(r-2)$-set $D$ and an $(r+2)$-set $S$ such that
$D\subset X_i,Y_j\subset S
    ~ (i,j\in\{0,1\})$,
  and $X_0,X_1$ and $Y_0,Y_1$ are the two pairs of opposite corners
  of a square in the interval $[D,S]$.  We call this the
  \emph{square type}.
\end{enumerate}
In the star case $|X_0\cap X_1|=r-1$, while in the square case
$|X_0\cap X_1|=r-2$.
\end{lemma}

For any graph $H$, the 2-blow-up $H[E_2]$ is obtained from $H$ by replacing every vertex $x$ by two independent vertices $x_0,x_1$ and every edge $xy$ by a complete bipartite graph with parts $\{x_0,x_1\}$ and $\{y_0,y_1\}$. The next lemma states a copy of $P_{H[E_2]}$ in the union of two consecutive levels has the expected form.

\begin{lemma}\label{lem:rigidity}
Let $H$ be a connected graph with at least two adjacent edges, and put
$G=H[E_2]$.  Suppose that $P_G$ is weakly embedded in two
consecutive levels $\binom{[n]}r\cup\binom{[n]}{r+1}$.  If
$A_v\in\binom{[n]}{r}$ is the image of $v\in V(G)$, then there is a
single $(r-1)$-set $R$ such that $A_v=R\cup\{a_v\}$ for every $v\in V(G)$,
where the elements $a_v$ are pairwise distinct.
\end{lemma}

\begin{proof}
Every vertex of $G$ has positive degree.  Hence every vertex-element
of $P_G$ must be mapped to the lower level and every edge-element to
the upper level.  If $uv\in E(G)$, then
$|A_u\cap A_v|=r-1$,
and the image of the edge $uv$ is necessarily $A_u\cup A_v$.

For every edge $xy\in E(H)$, the two clone-pairs
$\{x_0,x_1\}$ and $\{y_0,y_1\}$ span a $K_{2,2}$ in $G$.
Lemma~\ref{lem:block} says that its four lower images form either a
star block or a square block.  If two edges of $H$ share a vertex,
the corresponding blocks share a clone-pair.  In a star block the
two sets of this clone-pair intersect in $r-1$ elements, while in a
square block they intersect in $r-2$ elements.  Thus two blocks
sharing a clone-pair have the same type.  Since the line graph of a
connected graph with at least two edges is connected, all blocks have
the same type.

They cannot all be square blocks.  Indeed, let $xy,xz\in E(H)$.  The
two corresponding blocks share $A_{x_0},A_{x_1}$.  If they are square
blocks, then both have $D=A_{x_0}\cap A_{x_1},~
  S=A_{x_0}\cup A_{x_1}$.
For either block, the four upper edge-images are exactly the four
$(r+1)$-sets $U$ satisfying $D\subset U\subset S$.
The two blocks would therefore use the same four upper sets as the
images of eight distinct edge-elements of $P_G$, contradicting
injectivity.

Thus all blocks are of star type.  The stem of a star block containing
a given clone-pair is the intersection of the two sets in that pair.
Consequently adjacent blocks have the same stem.  Connectedness of
the line graph now implies that all lower images have one common
$(r-1)$-stem $R$.  Injectivity gives the distinctness of the elements
$a_v$.
\end{proof}

The next proposition states that for appropriately chosen $H$, the poset $P_{H[E_2]}$ satisfies $1=e(P_{H[E_2]})<\pi^-(P_{H[E_2]})$.

\begin{proposition}\label{prop:construction}
Let $q$ be a prime power, let $H$ be a connected graph satisfying $\chi(H)>q+1$,
and put $G=H[E_2],~ P=P_G$.
Then $e(P)=1$ and $\pi^-(P)\geq 1+\Delta_q$,
where $\Delta_q:=\prod_{j=1}^{\infty}(1-q^{-j})>0$.
\end{proposition}

\begin{proof}
The natural map $v\mapsto\{v\}$ and $uv\mapsto\{u,v\}$ embeds $P_G$ into two consecutive levels. Since $P_G$ is not an antichain, $e(P)=1$. Put $r=\lfloor n/2\rfloor$.

For the density construction, label
the elements of $[n]$ independently and uniformly by vectors $\lambda(x)\in\mathbb F_q^{\,\lfloor \frac{n}{2}\rfloor+1}$. Consider the family
\[
  \F=
  \binom{[n]}{\lfloor \frac{n}{2}\rfloor}
  \cup
  \left\{
    B\in\binom{[n]}{\lfloor \frac{n}{2}\rfloor+1}:
  (\lambda(x))_{x\in B}\text{ is linearly independent}
  \right\}.
\]

We claim that $\F$ is $P$-free.  Suppose otherwise.  By
Lemma~\ref{lem:rigidity}, the lower images of the vertex-elements of
$P_G$ have the form $R\cup\{a_v\},~ v\in V(G)$,
for one common $(\lfloor \frac{n}{2}\rfloor-1)$-set $R$.  Put $W=\Span\{\lambda(x):x\in R\}$.

If $uv\in E(G)$, then the upper image of the corresponding
edge-element is $R\cup\{a_u,a_v\}$, and this set belongs to $\F$ only if
its $\lfloor \frac{n}{2}\rfloor+1$ labels are linearly independent.  It follows that
$\dim W=r-1$, that in the quotient space $\mathbb{F}_q^{\lfloor \frac{n}{2}\rfloor+1}/W$ every $\lambda(a_v)+W$ is nonzero, and that  $\lambda(a_u)+W,\ \lambda(a_v)+W$
are linearly independent whenever $uv\in E(G)$.  Assigning to $v$ the
projective point
\[
  \bigl\langle\lambda(a_v)+W\bigr\rangle
  \in \PG(\mathbb F_q^{\lfloor \frac{n}{2}\rfloor+1}/W)=\PG(1,q)
\]
therefore gives a proper coloring of $G$ with $q+1$ colors.

On the other hand, $\chi\bigl(H[E_2]\bigr)=\chi(H)>q+1$, a contradiction.  Hence $\F$ is $P$-free.

The probability that $\lfloor \frac{n}{2}\rfloor+1$ vectors chosen independently and uniformly from
$\mathbb F_q^{\lfloor \frac{n}{2}\rfloor+1}$ are linearly independent is
$\Delta_{q,r+1}:=\prod_{j=1}^{r+1}(1-q^{-j})\rightarrow \Delta_q$.
By averaging, some labeling satisfies
\[
  |\F|
  \geq
  \binom{n}{\lfloor \frac{n}{2}\rfloor}+\Delta_{q,\lfloor \frac{n}{2}\rfloor+1}\binom{n}{\lfloor \frac{n}{2}\rfloor+1}.
\]
Thus we obtain $\pi^-(P)\geq 1+\Delta_q$.
\end{proof}

Now we are ready to show our infinite set of minimal posets. We use two known results.  First, Erd\H{o}s proved that graphs of
arbitrarily large girth and chromatic number exist \cite{Erdos}.  

\begin{theorem}[Griggs, Lu \cite{GriggsLu}]
    For any graph $G$, we have $\La(n,P_G)\le (1+\sqrt{1-\frac{1}{\chi(G)-1}}+o(1))\binom{n}{\lfloor \frac{n}{2}\rfloor}$. In particular, if $F$ is a nonempty bipartite graph, then
\begin{equation}\label{eq:GL}
  \La(n,P_F)
  =(1+o(1))\binom{n}{\lfloor n/2\rfloor}.
\end{equation}
\end{theorem}

\medskip

\begin{proof}[Proof of Theorem \ref{infiniteminimal}]
For every positive integer $m$, choose a connected graph $H_m$ such
that $g(H_m)>m, ~ \chi(H_m)>3$.
 Put $G_m=H_m[E_2], ~  P_m=P_{G_m}$.
Proposition~\ref{prop:construction}, applied with $q=2$, gives $e(P_m)=1, ~  \pi^-(P_m)\geq 1+\Delta_2>1$.

Among all posets $Q$ weakly contained in $P_m$ and satisfying $\pi^-(Q)>e(Q)=1$, choose one, say $Q_m$, that is minimal under weak
subposet containment.  By its
definition, $Q_m$ is globally minimal among the posets satisfying $\pi^-(Q)>e(Q)=1$: every weak subposet of $Q_m$ is also weakly contained
in $P_m$.

We claim that
\begin{equation}\label{eq:size}
  |Q_m|\geq g(H_m)>m.
\end{equation}
Let $Q$ be any weak subposet of $P_m$ such that $e(Q)=1, ~
  |Q|<g(H_m)$.
Fix a weak embedding $\varphi:Q\to P_{G_m}$.  Let $F$ be the 
subgraph of $G_m$ whose edge set consists of those edges $e\in E(G_m)$
for which the edge-element $e$ belongs to $\varphi(Q)$.  Then
\[
  |E(F)|\leq |Q|<g(H_m).
\]

The projection $\rho:G_m\rightarrow H_m,~
  \rho(x_i)=x$,
is a graph homomorphism.  If $F$ were nonbipartite, it would contain
an odd cycle $C$ of length at most $|E(F)|$.  The image $\rho(C)$
would be a closed odd walk in $H_m$.  Every closed odd walk contains
an odd cycle of no greater length, so $H_m$ would contain a cycle of
length at most $|C|\leq |E(F)|<g(H_m)$,
a contradiction.  Thus $F$ is bipartite.

Let $Q'$ be the weak subposet of $Q$ that we obtain by removing all isolated elements. The chosen map $\varphi$ is also a weak embedding of $Q'$ into $P_F$.
Indeed, whenever $x<y$ in $Q'$, the image $\varphi(x)$ is a
vertex-element of $P_{G_m}$, the image $\varphi(y)$ is an incident
edge-element, and this edge belongs to $F$ by construction.  Moreover, $F$ is nonempty: since $e(Q)=e(Q')=1$, the poset $Q$
is not an antichain and hence contains a relation.

By \eqref{eq:GL},
\[
  \binom{n}{\lfloor n/2\rfloor}\le \La(n,Q')\leq \La(n,P_F)
  =(1+o(1))\binom{n}{\lfloor n/2\rfloor}.
\]
Thus $\pi(Q')= 1$.  Observe that $\La(n,Q)\le \La(n,Q')+|Q\setminus Q'|$ for any $n$ as a weak copy of $Q'$ together with any further $|Q\setminus Q'|$ other sets form a weak copy of $Q$. So $\pi(Q)=1$, too. Therefore no weak subposet of
$P_m$ having fewer than $g(H_m)$ elements can satisfy both
$e(Q)=1$ and $\pi^-(Q)>1$.  In particular, \eqref{eq:size} holds.

The orders $|Q_m|$ are unbounded, so the sequence contains infinitely
many pairwise nonisomorphic posets.  Each of them is minimal with the
required properties, completing the proof.
\end{proof}

\section{Every fixed poset with \texorpdfstring{$e(P)=1$}{e(P)=1} has a
strict gap below 2}

In this section we prove Theorem \ref{thm:e1-gap}. The main ingredient is the following two-rank boundary lemma.

\begin{lemma}\label{lem:boundary}
Let $Q=\binom{[d]}{\ell}\cup \binom{[d]}{\ell+1}$, and
let $h=|Q|=\binom d\ell+\binom d{\ell+1}$.
Suppose that $P$ is a weak subposet of $Q$.  Let  $\mathcal L\subseteq\binom{[n]}r,  \mathcal U\subseteq\binom{[n]}{r+1}$,
where $r\ge\ell$ and $n-r\ge d-\ell$, and suppose that
$\mathcal L\cup\mathcal U$ is $P$-free.  Define
\[
  a=\frac{|\mathcal L|}{\binom nr},\qquad
  b=\frac{|\mathcal U|}{\binom n{r+1}},\qquad
  p=\frac{e(\mathcal L,\mathcal U)}{\binom nr(n-r)},
\]
where $e(\mathcal L,\mathcal U)$ is the number of pairs
$A\in\mathcal L$, $B\in\mathcal U$ with $A\subset B$.  Then, with  $\delta=\frac1{2(h-1)}$,
we have
\begin{equation}\label{eq:boundary}
  p\le \frac{1-\delta}{2}(a+b).
\end{equation}
\end{lemma}

\begin{proof}
Let $H$ be the bipartite inclusion graph between the two levels of $Q$.
This graph is connected: any two $\ell$-sets can be transformed into one
another by single-element exchanges, and each such exchange gives a path of
length two through an $(\ell+1)$-set.  The cases $\ell=0$ and $\ell=d-1$
are immediate.

Choose uniformly at random an $(r-\ell)$-set $C$ and an injection
$\theta:[d]\to[n]\setminus C$, and map
$X$ to $C\cup\theta(X)$.
By symmetry, the image of every fixed lower vertex of $Q$ is uniform in
$\binom{[n]}r$, the image of every fixed upper vertex is uniform in
$\binom{[n]}{r+1}$, and the image of every fixed edge of $H$ is uniform
among all cover edges between $\binom{[n]}{r}$ and $\binom{[n]}{r+1}$.

Mark a lower vertex when its image belongs to $\mathcal L$, and mark an
upper vertex when its image belongs to $\mathcal U$.  For a fixed edge of
$H$, the probability that exactly one endpoint is marked is
\[
  \beta=a+b-2p.
\]
Fix a spanning tree of $H$ and a lower root.  No standard copy can have all
its vertices marked, because a completely marked copy of $Q$ contains $P$.
Thus, whenever the root is marked, at least one of the $h-1$ tree edges has
exactly one marked endpoint.  Taking probabilities and using the union bound
gives
\[
  a\le(h-1)\beta.
\]
Using an upper root gives $b\le(h-1)\beta$.  Hence
\[
  \beta\ge\frac{\max\{a,b\}}{h-1}
  \ge\frac{a+b}{2(h-1)}=\delta(a+b).
\]
Since $2p=a+b-\beta$, this is exactly \eqref{eq:boundary}.
\end{proof}

\begin{proof}[Proof of Theorem \ref{thm:e1-gap}]
Because $e(P)=1$, there are fixed $d$ and $0\le\ell<d$ such that $P$ is a
weak subposet of
\[
  Q=\binom{[d]}\ell\cup\binom{[d]}{\ell+1}.
\]
Use this embedding to partition the elements of $P$ into a lower part of
size $s$ and an upper part of size $t$.  Every relation of $P$ goes from the
lower part to the upper part.  Isolated elements, if present, may be placed
in either part.  Since $e(P)=1$, both $s$ and $t$ are positive.

Let $\mathcal F\subseteq 2^{[n]}$ be $P$-free. As the number of subsets $G$ of $[n]$ with $||G|-n/2|>2\sqrt{n \log n}$ is $o(\binom{n}{\lfloor \frac{n}{2}\rfloor})$, we can assume that for all $F\in\F$, we have $||F|-n/2|\le 2\sqrt{n\log n}$.

Let
$\mathcal C$ be a uniformly random maximal chain, set
\[
  X=|\mathcal F\cap\mathcal C|,
  \qquad
  \mu=\mathbb EX
  =\sum_{F\in\mathcal F}\binom n{|F|}^{-1},
\]
and as the smallest possible summand in $\mu$ is $\frac{1}{\binom{n}{\lfloor \frac{n}{2}\rfloor}}$, we obtain $|\mathcal F|\le\mu\binom{n}{\lfloor \frac{n}{2}\rfloor}$. Thus it is enough to prove $\mu\le 2-\delta_P+o(1)$.

By convexity, we have $\frac{\mu(\mu -1)}{2}=\binom{\mathbb{E}(X)}{2}\le \mathbb{E}\binom{X}{2}$, so our aim is to bound
\[
\mathbb{E}\binom{X}{2}=\sum_{A\subset B,~  A,B\in \F}\frac{|A|!(|B|-|A|)!(n-|B|)!}{n!}=
\]
\[
\sum_{A\subset B, ~|B|=|A|+1,~A,B\in\F}\frac{|A|!(n-|A|-1)!}{n!}+\sum_{A\subset B,~ |B|-|A|\ge 2, ~ A,B\in \F}\frac{|A|!(|B|-|A|)!(n-|B|)!}{n!}=:W_{1}+W_{2}
\]

We first show $W_{2}=o(\mu)$.  Let $Y$ count
triples $A\subsetneq S\subsetneq B$, $A,B\in\mathcal F$,
on the random chain $\mathcal C$, where $S$ need not belong to $\mathcal F$. Clearly, $W_{2}\le \mathbb EY$.  For a fixed $S$, put
\[
  u(S)=\sum_{\substack{A\in\mathcal F\\A\subsetneq S}}
          \binom{|S|}{|A|}^{-1},
  \qquad
  v(S)=\sum_{\substack{B\in\mathcal F\\S\subsetneq B}}
          \binom{n-|S|}{|B|-|S|}^{-1}.
\]
Observe that $u(S)$ is the expected number of sets below $S$ in $\F$ on a maximal chain from $\emptyset$ to $S$ taken uniformly at random, and $v(S)$ is the expected number of sets above $S$ in $\F$ on a maximal chain from $S$ to $[n]$ taken uniformly at random. As the probability that $S\in \mathcal C$ is $\binom{n}{|S|}^{-1}$, we have 
\begin{equation*}\label{eq:Y}
  \mathbb EY=\sum_S\binom n{|S|}^{-1}u(S)v(S),
\end{equation*}
where the summation can be restricted to sets $S$ with $||S|-n/2|\le 2\sqrt{n\log n}$ as other sets do not lie between two sets of $\F$. 

For every $S$, either fewer than $s$ members of $\mathcal F$ lie strictly
below $S$, or fewer than $t$ lie strictly above $S$.  Otherwise, those lower
and upper sets form a copy of $K_{s,t}$ containing a weak copy of
$P$.  In the first case,
\[
  u(S)\le\frac{s-1}{n/2-2\sqrt{n\log n}},
\]
because $\binom{|S|}{|A|}\ge |S|\ge n/2-2\sqrt{n\log n}$ for every strict $A\subset S$;
the second case gives the dual bound
\[
  v(S)\le\frac{t-1}{n/2-2\sqrt{n\log n}}.
\]
Furthermore, as all sets $F\in \F$ have size satisfying $||F|-n/2|\le 2\sqrt{n\log n}$, we obtain
\[
  \sum_S\binom n{|S|}^{-1}u(S)=\sum_{F\in\F}\binom{n}{|F|}^{-1}\sum_{S:F\subsetneq S, ~|S|\le n/2+2\sqrt{n\log n}}\binom{n-|F|}{|S|-|F|}^{-1}
  \le4\sqrt{n\log n}\mu,
\]
and the dual identity gives the same bound with $v(S)$ in place of $u(S)$.
Partitioning the sets $S$ according to which of the two  bounds above
holds, we obtain
\begin{equation}\label{eq:long-pairs}
  W_{2}
  \le\mathbb EY
  \le\frac{4\sqrt{n\log n}(s+t-2)}{n/2-2\sqrt{n\log n}}\,\mu
  =o(\mu).
\end{equation}

For $F\in\mathcal F$, let $d^+(F)=|\{F'\in\F:F\subset F',|F'|=|F|+1\}|$ and $d^-(F)=|\{F'\in\F:F'\subset F,|F'|=|F|-1\}|$, and define
\[
  \mathcal F_1=\{F\in\mathcal F:d^+(F)\ge t\},
  \qquad
  \mathcal F_2=\{F\in\mathcal F:d^-(F)\ge s\}.
\]
These two families are disjoint.  Indeed, a common member, together with
$s$ immediate predecessors and $t$ immediate successors, forms a copy of $K_{s,t}$ containing $P$.

We partition $W_{1}$ into three subsums
\[
W_{1}=\sum_{A\subset B, ~|B|=|A|+1,~A,B\in\F,~A\not\in \F_1}\frac{1}{n-|A|}\binom{n}{|A|}^{-1}+\sum_{A\subset B, ~|B|=|A|+1,~A,B\in\F,~A\in\F_1,~B\not\in\F_2}\frac{1}{|B|}\binom{n}{|B|}^{-1}
\]
\[+\sum_{A\subset B, ~|B|=|A|+1,~A\in  \F_1,B\in\F_2}\frac{1}{n-|A|}\binom{n}{|A|}^{-1}
\]

By definition of $\F_1$ and $\F_2$, the first two sums are at most $\frac{t-1}{n/2-2\sqrt{n\log{n}}}\,\mu$ and $\frac{s-1}{n/2-2\sqrt{n\log{n}}}\,\mu$, respectively.
For the third sum, apply Lemma~\ref{lem:boundary} for each $r$ with $n/2-2\sqrt{n\log n}\le r\le n/2+2\sqrt{n\log n}$,  to
\[
  \mathcal L=\mathcal F_1\cap\binom{[n]}r,
  \qquad
  \mathcal U=\mathcal F_2\cap\binom{[n]}{r+1}.
\]
Observe that the summand $\frac{1}{n-|A|}\binom{n}{|A|}^{-1}=\frac{1}{n-r}\binom{n}{r}^{-1}$ in $W_{1}$ is exactly the denominator of $p$ in Lemma \ref{lem:boundary}. Summing
\eqref{eq:boundary} over $r$ and using
$\mathcal F_1\cap\mathcal F_2=\varnothing$ gives
\begin{equation}\label{eq:adjacent-pairs}
  W_{1}
  \le\frac{1-\delta_P}{2}
       \left(\sum_{A\in\mathcal F_1}\binom n{|A|}^{-1}
       +\sum_{B\in\mathcal F_2}\binom n{|B|}^{-1}\right)+o(\mu)
  \le\frac{1-\delta_P}{2}\mu+o(\mu).
\end{equation}

Equations \eqref{eq:long-pairs} and \eqref{eq:adjacent-pairs} yield
\[
  \mathbb E\binom X2
  \le\frac{1-\delta_P}{2}\mu+o(\mu).
\]
Therefore
$\mathbb E\binom X2
  \ge\binom{\mathbb EX}{2}
  =\frac{\mu(\mu-1)}2$, and
thus $\mu\le2-\delta_P+o(1)$ as needed.  
\end{proof}

\begin{remark}
    The value of 2 in Theorem \ref{thm:e1-gap} is sharp. Consider the vertex-edge incidence poset $P_{K_t}$. The Ellis-Ivan-Leader construction in \cite{EllisIvanLeader} gives an $(\lfloor \frac{n}{2}\rfloor,t)$-daisy free family $\F\subseteq \binom{[n]}{\lfloor \frac{n}{2}\rfloor}$ of size $(1-\frac{1}{t}-o(\frac{1}{t}))\binom{n}{\lfloor \frac{n}{2}\rfloor}$. So $\binom{[n]}{\lfloor \frac{n}{2}\rfloor-1}\cup \F$ is $P_{K_t}$-free of size $(2-\frac{1}{t}-o(\frac{1}{t}))\binom{n}{\lfloor \frac{n}{2}\rfloor}$. 
\end{remark}

\section{Unbounded lower density at
\texorpdfstring{$e(P)=2$}{e(P)=2}}

In this section we prove Theorem \ref{thm:e2} that states that for any $K>0$ there exists a poset $P$ with $e(P)=2$ and $\pi^-(P)>K$.

\begin{proof}[Proof of Theorem \ref{thm:e2}]
Fix an integer $L>2K$.  Choose $L$ consecutive levels $a,a+1,\dots,b$ with $b=a+L-1$,
centered at $n/2$.  Label the elements of $[n]$
independently and uniformly by vectors $\phi(x)\in V:=\mathbb F_2^{b+1}$.
Retain precisely those sets in the chosen levels whose label vectors are linearly independent:
\[
  \F_\phi=
  \bigl\{A\subseteq[n]:a\le|A|\le b,
  \ (\phi(x))_{x\in A}\text{ is linearly independent}\bigr\}.
\]
For a fixed $s$-set with $s\le b$, Lemma~\ref{lem:random} gives
\[
  \Pr(A\notin\F_\phi)
  \le\sum_{i=0}^{s-1}2^{i-b-1}
  =\frac{2^s-1}{2^{b+1}}
  <2^{s-b-1}\le\frac12.
\]
By linearity of expectation some labelling satisfies
$|\F_\phi|>\frac12\sum_{s=a}^{b}\binom ns
=\left(\frac L2+o(1)\right)\binom{n}{\lfloor n/2\rfloor}$.
Since $L/2>K$, it follows that
$\pi^-\bigl(S_2(m,0)\bigr)\ge L/2>K$.

It remains to prove that every such $\F_\phi$ is $S_2(m,0)$-free when
$m=2^{L^2}+1$.  Suppose instead that $\psi$ is a weak copy, and write
\[B=\psi(\varnothing),\qquad
  B_i=\psi(\{i\}),\qquad
  C_{ij}=\psi(\{i,j\}).
\]
Let $W=\Span\{\phi(x):x\in B\}$ and $H_i=\Span\{\phi(x):x\in B_i\}$.
Since all retained sets have independent indexed labels, $\dim W=|B|\ge a$ and $W\le H_i$.
Thus  $\operatorname{codim}_V W\le L$.
By Lemma~\ref{lem:superspaces}, at most $2^{L^2}=m-1$ different subspaces
$H_i$ can contain $W$.  Hence $H_i=H_j$ for some $i\ne j$.
Independence of $B_i$ and $B_j$ now gives
$|B_i|=\dim H_i=\dim H_j=|B_j|$.

The sets $B_i$ and $B_j$ are distinct because $\psi$ is injective, so
$|B_i\cup B_j|>|B_i|$.  Every label indexed by $B_i\cup B_j$ lies in the
common $|B_i|$-dimensional space $H_i$.  Those indexed labels are therefore
dependent.  But $B_i\cup B_j\subseteq C_{ij}$,
so the labels indexed by $C_{ij}$ are dependent as well, contradicting
$C_{ij}\in\F_\phi$.  This proves the claim.  
\end{proof}

\bigskip

\noindent\textbf{AI declaration}.  The constructions in this manuscript were obtained through the use of {ChatGPT~5.6}.  The arguments have been reworked and carefully verified by the authors, who take full responsibility for the content of the manuscript.

\end{document}